\documentclass[review]{elsarticle}

\usepackage{lineno,hyperref}
\modulolinenumbers[5]
\usepackage{amsmath,amsfonts,amsthm,amssymb,dsfont,color,graphicx,subfigure,caption,url}
\usepackage{subfigure,multirow,booktabs,latexsym}  
\def\be{\begin{equation}}
\def\ee{\end{equation}}
\def\sign{\hbox{\rm sign}\,}
\def\abs#1{\left\vert #1 \right\vert} 
\newtheorem{lemm}{Lemma}[section]
\newtheorem{defi}{Definition}[section]
\newtheorem{thm}{Theorem}[section]
\newtheorem{prop}{Proposition}[section]
\newtheorem{lemma}{Lemma}[section]

\newtheorem{algo}{Algorithm}[section]

\def\l1{L1}  %
\def\bp{\begin{prop}}
\def\ep{\end{prop}}
\def\bpf{\begin{proof}}
\def\epf{\end{proof}}
\def\bd{\begin{defi}}
\def\ed{\end{defi}}
\def\br{\begin{oss}}
\def\er{\end{oss}}
\def\bl{\begin{lemma}}
\def\el{\end{lemma}}

\begin{document}

\begin{frontmatter}

\title{A Fast Splitting Method for efficient  Split Bregman Iterations \tnoteref{mytitlenote}}
\tnotetext[mytitlenote]{Method for  \href{http://www.ctan.org/tex-archive/macros/latex/contrib/elsarticle}{CTAN}.}

\author[a]{D. Lazzaro}

\author[b]{E. Loli Piccolomini}
\cortext[mycorrespondingauthor]{Corresponding author}
\ead{elena.loli@unibo.it}

\author[a]{F. Zama}

\address[a]{Department of Mathematics, University of Bologna,
Piazza di Porta San Donato, 5, 40126, Bologna (ITALY)}
\address[b]{Computer Science and Engineering Department, University of Bologna,
Mura Anteo Zamboni, 7, 40126, Bologna (ITALY)}

\begin{abstract}
In this paper we propose a new fast splitting algorithm to solve the Weighted Split Bregman minimization problem
in the backward step of an accelerated Forward-Backward algorithm. Beside proving the  convergence of the method, 
numerical tests, carried out on different imaging applications, prove the accuracy and computational efficiency of the proposed algorithm.
\end{abstract}


\end{frontmatter}

\linenumbers
%
\section{Introduction\label{FBstep}}
A large number of important image processing applications require the solution of a regularized optimization problem. In order
to cope with the inner ill-conditioning of the model and its  sensitivity to noise, a data fit term is balanced by a weighted regularization term.
Among the different regularization functions, the Total Variation (TV) and the Weighted Total Variation (WTV) have recently
gained increasing attention because of their edge preserving properties \cite{ROF92,Zhang2010}. Therefore we focus on 
the numerical solution of the regularized minimization problem:
\begin{equation}
\min_{u} \left \{ f(u) + \lambda \ WTV(u) \right \}
 \label{eq:RegProb}
\end{equation}
where $f(u)$ is the least squares fit term,  $WTV(u)$ is the weighted total variation regularization term and $\lambda>0$ is the regularization parameter.
The choice of the weighting function in $WTV$ is crucial   to filter out noise while preserving the image edges. In this paper we apply the non-convex 
{\it log-exp} function,  proposed in \cite{Montefusco2013}. The solution of problem \eqref{eq:RegProb} is tackled  by an Accelerated Forward Backward algorithm 
where a modified FISTA acceleration strategy \cite{Chambolle2015}  is applied to the Backward step. The Weighted Split Bregman (WSB) method, used to
compute the Backward step, generates a sequence of inner linear systems which constitute the computational  core of  the whole algorithm. In this work we propose 
an iterative solver (FWSB) based on a new matrix splitting which uses   the matrices structure to achieve accurate and efficient solutions. 
Besides proving the convergence of our iterative method, we compare it to the Gauss Seidel solver on different imaging problems. The tests 
confirm its better performances in terms of accuracy and computation times.

The present paper is organized as follows: in section \ref{S1} we present the accelerated Forward-Backward algorithm.
In section \ref{A_B} the details of the Backward steps are examined  and in section \ref{SPL} the new splitting algorithm is introduced 
and its convergence proven. Finally the numerical results and conclusions are reported in sections \ref{NumRes}, \ref{Con} respectively.

\section{The Accelerated Forward Backward Algorithm \label{S1}}

In this section we introduce our WTV function as the sum of $L_1$ norm  of the weighted gradient $(\nabla_x^w, \nabla_y^w)^T$  of the image $u$ along the coordinate directions:
$$  WTV(u)= \|\nabla_x^w u\|_1+\|\nabla_y^w u\|_1 $$
where:
\begin{equation}
\|  \nabla_x^w u\|_1=\sum_{i,j=1}^N  (w_{i,j}^x) | u_x|_{i,j}, \ \ \|  \nabla_y^w u\|_1=\sum_{i,j=1}^N   (w_{i,j}^y) | u_y|_{i,j}
\label{pesi}
\end{equation}
and $w_{i,j}^x>0$ and $w_{i,j}^y>0$ are constants that weight the first order differences $u_x$ and $u_y$, along the vertical and horizontal directions 
respectively.
The choice of the weights $w^x$ and $w^y$ is crucial in our approach.
In order to preserve the image edges, the weight for a pixel $u_{i,j}$ can be chosen to be  inversely proportional to 
the local value of the gradient. Therefore it is small when the gradient of the image is large, hence when there is an edge, 
and it is large when the gradient of the image is small, hence in locations corresponding to uniform areas 
where small variations are mainly due to the presence of noise. We define  the weights of the WTV,  at each pixel, as the derivative of a
strongly non-convex function of the gradient of the noisy  image $u$ in the same pixel.
In particular we choose  the derivative of the non-convex {\it log-exp} function $\phi_{\mu}(t)$  as proposed in \cite{Montefusco2013}, 
 \begin{equation}
\label{deffi}
\phi_{\mu}(|t|)=\frac{1}{\log 2} \log \left ( \frac{2}{1+e^{-\frac{|t|}{\mu} }} \right), \quad \mu>0,
\end{equation}
whose derivative is given by
 \begin{equation}
\label{deffip}
\phi'_{\mu}(|t|)=\frac{1}{\mu \cdot log(2)} \frac{1}{1+e^{\frac{|t|}{\mu}}}
\end{equation}
and satisfies:
\begin{equation}
\left\{
\begin{tabular}{ccc}
$\phi_{\mu}'(|t|)\rightarrow 0$& for $|t|>\mu$&\\
& &$\mu>0$\\
$\phi'_{\mu}(|t|)$ large &for $|t| <\mu$\\
\end{tabular}
\right. 
\end{equation}
$\phi'_{\mu}(|\cdot|)$ approaches to zero near the edges, where the gradient gets large, 
while it is large in smooth areas where the gradient becomes small. 
In fact, besides separating edges from smooth areas, $\phi'_{\mu}(|\cdot|)$ also identifies the small differences in intensity variations within the smooth areas. 
\par
Since we adopt anisotropic TV discretization, our weights $w_{i,j}^x$ and $w_{i,j}^y$ are different along the $x$ and $y$ directions and are given by:
\begin{align}
 w^x_{i,j}=\phi'_{\mu}(|u_x|_{i,j})= \frac{1}{\mu \cdot log(2)} \frac{1}{1+e^{\frac{|u_x|_{i,j}}{\mu}}}, \label{pesix}\\ 
  w^y_{i,j}= \phi'_{\mu}(|u_y|_{i,j})=\frac{1}{\mu \cdot log(2)} \frac{1}{1+e^{\frac{|u_y|_{i,j}}{\mu}}}.\label{pesiy}
\end{align}

By setting the data fit function $f$ as the least squares distance from the data $z$, we have
\begin{equation}
f(u) = \frac{1}{2 }\| \Phi u - z \|_2^2
\label{eq:model}
\end{equation}
where $\Phi$ is an $m \times n$ linear operator used to model different applications. It can be 
 a convolution operator in the deblurring problem or a subsampling measurement operator in the compressive sensing problem, etc.

Finally we define our problem as follows:
\begin{equation}
\hbox{find } \ u^*, \ \ s.t. \ \ 
u^*= \arg \min_{u} \left \{ \frac{1}{2 }\| \Phi u - z \|_2^2+ \lambda \left(\|\nabla_x^w u\|_1+\|\nabla_y^w u\|_1 \right) \right \}.
\label{eq:convex}
\end{equation}

Problem \eqref{eq:convex} is convex and non differentiable and  it has a unique solution under the trivial hypothesis of $\Phi \neq 0$.\\Different methods can be used for its solution such as  Chambolle Pock  \cite{Chambolle10},  Split-Bregman \cite{Oscher2009}, Alternating Minimization  \cite{Wang2008}. All the methods should converge to the same point, with different rate.   In this paper we use the Forward-Backward (FB) algorithm for the solution of the convex minimization problem \eqref{eq:convex}, since it requires the tuning of very few parameters which is a great advantage in real applications. 
\\
We solve \eqref{eq:convex} by a converging sequence of Accelerated Forward-Backward steps 
$(v^{(n)}, u^{(n)}), n=1,2, \ldots$ where a modified FISTA  acceleration strategy \cite{Chambolle2015} 
is applied to the backward step. Given $u^{(0)}$, we compute for $n=1,2 \ldots$: 
\begin{flalign}
v^{(n)}={u}^{(n-1)} +\beta \Phi^t \left (z - \Phi{u}^{(n-1)} \right ) \\
{\tilde u}^{(n)} = \arg\min_{u} \left \{\lambda  \left (\|  \nabla_x^w u\|_1+ \|  \nabla_y^w u\|_1 \right ) +\frac{1}{2  \beta} 
\| u - v^{(n)} \|_2^2 \right \}\label{B} 
\end{flalign}

 \begin{equation}
u^{(n)}={\tilde u}^{(n-1)}+\alpha({\tilde u}^{(n)}-{\tilde u}^{(n-1)}), 
 \label{eq:F}
 \end{equation}
and $\alpha$ is chosen as follows:
 \begin{equation}
\alpha=\frac{t_{n-1}-1}{t_{n}},  \ \ 
t_n = \frac{n+a+1}{a}.
\label{eq:alpha}
\end{equation}
%
In our experiments, we set $a=2$.
In order to ensure the  convergence of the sequence $(v^{(n)}, u^{(n)})$ to the solution of \eqref{eq:convex}, 
the following condition on $\beta$ must hold \cite{Chambolle2015}:
$$ 0 < \beta < \frac{1}{\lambda_{max}(\Phi^t\Phi)}$$ 
where $\lambda_{max}$ is the maximum eigenvalue in modulus.
The Forward-Backward iterations are stopped with  the following  stopping condition:
\begin{equation}
\frac{\|u^{(n)}-u^{(n-1)}\|_2}{\|u^{(n)}\|_2} < \epsilon 
\ \ \ \hbox{where} \ \ \ 
\epsilon>0.
\label{stop}
\end{equation}

We observe that while $v^{(n)}$ and  $u^{(n)}$ are computed by explicit formulae, for the computation of ${\tilde u}^{(n)}$ in \eqref{B},
 we introduce a Split-Bregman strategy (section \ref{A_B}) and we propose a modified matrix splitting in the solution of the arising linear system.
 
The steps of the Accelerated Forward backwards Algorithm (AFB)  are reported in algorithm \ref{alg2}.
\vspace{5mm}
\begin{table}[!h]
\begin{algo}[\textsc{ Algorithm AFB}]
\begin{center}
\begin{tabular}{c}\hline
 \textbf{Input: $r_0,z,\beta$, $w^x$, $w^y$ Output: $u^*$
}\quad \quad \quad \quad\quad \quad \quad \quad    \\ 
 \hline 
\begin{minipage}[c]{11cm}
\begin{tabular}{ll}
$ u^{(0)}=\phi^T z, \, w^x=1,\, 
w^y=1, \,\lambda_0=r_0\|u^{(0)}\|_1, \mu_0=\| \nabla u^{(0)}\|_1$\\[1pt]
$\tilde u^{(0)}=u^{(0)}=u^{(0)}$;\\[1pt] 
\quad  $n=0$\\
\quad {\tt repeat} \\[2pt]
  \quad\quad   $ {v^{(n+1)}}=u^{(n)}+ \beta \Phi^T(z-\Phi u^{(n)}) $ (Forward Step)  \\
 \quad \quad Compute ${\tilde u}^{(n+1)}$ by solving \eqref{B}\\ 
\quad\quad compute $\alpha$ as in \eqref{eq:alpha}  \\
\quad \quad $u^{(n+1)}={\tilde u}^{(n+1)}+\alpha({\tilde u}^{(n+1)}-{\tilde u}^{(n)})$\\[2pt]
 \quad\quad $ n=n+1$ \\[2pt]
  \quad {\tt until stopping condition}  as in \eqref{stop} \\[2pt]
$u^*=u^{(n)}$
\end{tabular}
\end{minipage}\\ \hline
\end{tabular}
\end{center}
 \label{alg2}
 \end{algo} 
\caption{Accelerated Forward Backward Algorithm}
\end{table}

%
%
We point out that  the minimization problem \eqref{B} can be  efficiently solved by means of different methods existing in literature.
We cite, among others, \cite{Oscher2009} and \cite{Wang2008}. In this paper, we use  
a splitting variable strategy, proposed in  \cite{Zhang2010}. 

\section{The Weighted Split Bregman Method \label{A_B}}

In this section we recall  the Split Bregman method for Weighted Total Variation 
for the solution of \eqref{B}.\\
Introducing two auxiliary vectors $D_x, D_y \in \mathbb{R}^{N^2}$ we rewrite \eqref{B} as 
a constrained minimization problem as follows:
\begin{equation}
\label{SB}
 \min_{u} \left \{\frac{1}{2 \beta }\| u - v^{(n)} \|_2^2 +  \lambda \left(\| D_x\|_1 + \| D_y \|_1 \right) \right \}, 
 \ \ \ s.t. \ \ D_x = \nabla_x^w u, \ \ D_y = \nabla_y^w u,
\end{equation}
where $\| D_x\|_1$ and $\| D_x\|_2$ are defined as in \eqref{pesi}.
Hence \eqref{SB} can be stated in its quadratic penalized form as:
\begin{equation}
\min_{u, D_x, D_y} \left \{\frac{1}{2 \beta } \| u - v^{(n)} \|_2^2 +  \lambda \left (\| D_x\|_1 + \| D_y \|_1 \right)  + 
\frac{\theta}{2}
\left (  \|  D_x-\nabla_x^w u \|_2^2 +  \| D_y-\nabla_y^w u \|_2^2 \right ) \right \}
\label{eq:LSB}
\end{equation}
where $ \theta>0$ represents the penalty parameter.
In order to  simplify the notation, exploiting the symmetry in the $x$ and $y$ variables, 
we use the subscript $q$  indicating either $x$ or $y$.

By applying the Split Bregman iterations, given an initial iterate $U^{(0)}$, we compute a sequence $U^{(1)},U^{(2)}, \ldots, U^{(j+1)}$ 
by splitting \eqref{eq:LSB}  into three minimization problems as follows. 

Given $e_q^{(0)}=0$, $D_q^{(0)}=0$ and $U^{(0)}={v}^{(n)}$, compute:
\begin{multline}
\label{UK}
 U^{(j+1)}=\arg \min_{u }\left\{ \frac{1}{2 \beta} \| u -  {v}^{(n)} \|_2^2+ 
 \frac {\theta}{2}  \| D^{(j)}_x-\nabla_x^w u -e_x^{(j)}\|_2^2+ \right . \\
\left. \frac { \theta}{2} \| D_y^{(j)}-\nabla_y^w u -e_y^{(j)}\|_2^2\right \}
\end{multline}
\begin{multline}
\label{Dx}
D_{q}^{(j+1)}=\arg \min_{ D_{q} }\left\{\lambda \|D_{q}\|_1+\frac {\theta}{2}  \| D_{q}-\nabla_{q}^w U^{(j+1)}-e_{q}^{(j)} \|_2^2 \right \}\\
= \mbox{Soft}_{\Lambda} (\nabla_{q}^w U^{(j+1)}+e_{q}^{(j)}),
\end{multline}
where 
\begin{equation}
\Lambda = \frac{\lambda}{\theta}
\label{eq:Lambda}
\end{equation}
and  $e_{q}^j$ is updated according to the following equation:
\begin{multline}
\label{ex}
e_{q}^{(j+1)}=e_{q}^{(j)}+\nabla_{q}^w U^{(j+1)}-D_{q}^{(j+1)}=e_{q}^{(j)}+\nabla_{q}^w U^{(j+1)}-\\ 
\mbox{Soft}_{\Lambda} (\nabla_{q}^w U^{(j+1)}+e_{q}^{(j)}) =\mbox{Cut}_{\Lambda} (\nabla_{q}^w U^{(j+1)}+e_{q}^{(j)}),
\end{multline}
 We remind that the Soft and the Cut operators apply point-wise  respectively as:
\begin{equation}
\label{Soft}
\mbox{Soft}_{\Lambda}(z)=\sign(z) \max \left\{ \abs{z}-{\Lambda},0\right \} 
\end{equation}
\begin{equation}
\label{Cut}
\mbox{Cut}_{\Lambda}(z)=z-\mbox{Soft}_{\Lambda} (z)= 
\left\{
\begin{tabular}{cc}
$\Lambda$&$z>{\Lambda}$\\
$z$&$-{\Lambda}\leq z \leq {\Lambda}$\\
$-{\Lambda}$&$z< -{\Lambda}$.\\
\end{tabular}
\right.
\end{equation}

By imposing first order optimality conditions in \eqref{UK},we compute the minimum $U^{(j+1)}$  by solving the following linear system
\begin{multline}
\label{SistL}
(\frac{1}{\beta}I-\theta \Delta^w)U^{(j+1)}= \frac{1}{\beta}v^{(n)} + \theta  (\nabla_x^w)^T(D_x^{(j)}-e_x^{(j)})+\\
 \theta (\nabla_y^w)^T(D_y^{(j)}-e_y^{(j)})
\end{multline}

where 
\begin{equation}
\Delta^w=- \left ((\nabla_x^w)^T \nabla_x^w+(\nabla_y^w)^T \nabla_y^w \right ).
\label{eq:lapl}
\end{equation}

Defining: 
\begin{equation}
A=(I-\beta \theta \Delta^w)
\label{eq:matr}
\end{equation}
and
\begin{equation}
b^{(j)}=v^{(n)} + \beta \theta  (\nabla_x^w)^T(D_x^{(j)}-e_x^{(j)})+\\
\beta \theta (\nabla_y^w)^T(D_y^{(j)}-e_y^{(j)})
 \label{eq:zj}
\end{equation}
the linear system \eqref{SistL} can be written as:
 \begin{equation}
 A U^{(j+1)} = b^{(j)},   
 \label{Slin1}
\end{equation}

 We observe that the solution of systems \eqref{Slin1} is a crucial point because it occurs in the inner loop of the backward step, therefore it is important to employ accurate and fast methods. 
Since the matrix $A$ is  sparse, strictly diagonally dominant and positive definite, the natural choice is to use the Gauss–Seidel or Conjugate Gradient  Methods. 

In the next paragraph we explain the details of our proposed method named Fast Weighted Split-Bregman (FWSB) to efficiently solve \eqref{Slin1}.  
\section{The proposed Matrix Splitting \label{SPL}}
In this paper, exploiting the structure of the matrix $A$ we obtain a matrix splitting of the form $E-F$ where $E$ is the Identity matrix
and $F$ is $\beta \theta \Delta^w$. We can prove that the iterative method, based on such a splitting, is convergent if 
\begin{equation}
 0 < \theta < \frac{1}{\beta \|\Delta^w \|. }
 \label{Theta}
\end{equation}
%

\begin{thm}
\label{th:t1}
 Let $E$ and $F$ define a splitting of the matrix $A=E-F$ in  \eqref{Slin1} as:
 \begin{equation}E = I, \ \ \ F= \beta  \theta \Delta^w.
\label{eq:split}
\end{equation}
  By choosing $\theta$ as in \eqref{Theta} we can prove that the spectral radius $\rho(E^{-1}F)<1$ and, for each right-hand side $B$, the following iterative method 
\begin{equation}
\label{solS}
X^{(m+1)}=F X^{(m)}+ B, \ \ m=0, 1, \ldots 
\end{equation}
 converges to the solution of the linear system $AX=B$.
\end{thm}


In order to prove theorem \ref{th:t1} we first prove the following lemma. 
\begin{lemm}
\label{lm:l1}
Let $M=E+F$ where $E$ is the Identity matrix and $F=\beta \theta \Delta^w $ with 
$\theta,\beta >0$. If $0< {\theta}<\frac {1}{ \beta \|\Delta^w\|_\infty}$
  then $M$ is a symmetric positive definite matrix. 
\end{lemm}
\begin{proof}
By definition of $\Delta^w$ in \eqref{eq:lapl},  it easily follows that $M=(I + \beta \theta \Delta^w)$
is a  real symmetric matrix. Moreover  in the finite discrete setting the $k$-th component of the product $\Delta^\omega u$ is:
\begin{multline}
\label{ElLapl}
\left (\Delta^\omega u \right )_{k}=-(\alpha_{k}^2+\alpha_{k+1}^2+\eta_{k}^2+\eta_{k+N}^2)u_{k}+\alpha_{k+1}^2u_{k+1}+ \\
\alpha_{k}^2 u_{k-1}+\eta_{k+N}^2u_{k+N}+\eta_{k}^2u_{k-N}
\end{multline}
where,  from equations \eqref{pesix} and \eqref{pesiy} it follows:
\begin{equation}
\alpha_k = \phi'_{\mu}(|u_x|_{i,j}) \ \ \eta_k = \phi'_{\mu}(|u_y|_{i,j}), \ \ k = (i-1)N+j 
\label{eq:weights}
\end{equation} 
Hence on each row there are at most  five non zero elements given by:\\
\begin{equation}
\label{definB}
\begin{tabular}{ll}
$M_{k,k}$=& ${1-}{\beta \theta}(\alpha_{k}^2+\alpha_{k+1}^2+\eta_{k}^2+\eta_{k+N}^2)$,\\
$M_{k,k-1}$=&${\beta  \theta}\alpha_{k}^2$, \\
$M_{k,k+1}$=&${\beta  \theta}\alpha_{k+1}^2$,\\
$M_{k,k-N}$=&${\beta  \theta}\eta_{k}^2$,\\
$M_{k,k+N}$=&${\beta \theta}\eta_{k+N}^2$\\
\end{tabular}
\end{equation}
In order to guarantee that the matrix $M$ is positive definite, 
it is sufficient to determine $\theta$
such that $M$ is strictly diagonally dominant, namely
\begin{multline*}
\abs {1-  \beta \theta (\alpha_{k}^2+\alpha_{k+1}^2+\eta_{k}^2+\eta_{k+N}^2) }>\beta  \theta
(\alpha_{k}^2+\alpha_{k+1}^2+\beta_{k}^2+\eta_{k+N}^2) \\ \quad \forall \,\, k=1,...,N^2
\end{multline*}
$$ \frac{1}{\beta \theta} > 2 (\alpha_{k}^2+\alpha_{k+1}^2+\eta_{k}^2+\eta_{k+N}^2) \quad \forall \,\, k=1,...,N^2$$ 
It easily follows that this relation is satisfied for 
$$\beta \theta <  \max_{k=1,..N^2}\left ( 2(\alpha_{k}^2+\alpha_{k+1}^2+\eta_{k}^2+\eta_{k+N}^2) \right )$$
namely, 
\begin{equation}
\label{CBDP}
 0< {\theta}<\frac {1}{\beta \|\Delta^w\|_{\infty}}
\end{equation}
\end{proof}
Proof of theorem \ref{th:t1}
\begin{proof}
 Using the Householder-Johns theorem \cite{A,B} $\rho(E^{-1}F)<1$ iff $A=E-F$ is symmetric positive definite (SPD) and $E^*+F$ 
 is symmetric and positive definite, where $E^*$ is the conjugate transpose of $E$. The matrix $A$ is SPD since it is symmetric and strictly diagonal dominant.
 From  Lemma \ref{lm:l1}, we have $E^*+F=M$ and therefore the condition on $\theta$ guarantees that $ E^*+F$ 
 is symmetric and positive definite.
\end{proof}

Hence we compute the Weighted Split Bregman solution $U^{(j+1)}, j=0,1,\ldots$ by means of the iterative method defined in \eqref{solS}
with $B = b^{(j)}$ as in \eqref{eq:zj}.

Substituting \eqref{eq:lapl} in \eqref{solS} we have:
\begin{equation}
X^{(m+1)}= - \beta  \theta  \left ((\nabla_x^w)^T \nabla_x^w+(\nabla_y^w)^T \nabla_y^w \right )X^{(m)}+ b^{(j)}.
\label{eq:eq_X}
\end{equation}
 By  substituting  \eqref{eq:zj} in \eqref{eq:eq_X} and  collecting $\nabla_x^w$, $\nabla_y^w$ we obtain :
\begin{multline}
 X^{(m+1)}= v^{(n)} + \beta \theta \left [ (\nabla_x^w)^T(-\nabla_x^w X^{(m)} + D_x^{(j)}-e_x^{(j)})+ \right . \\
 \left . +(\nabla_y^w)^T(-\nabla_x^w X^{(m)}+ D_y^{(j)}-e_y^{(j)}) \right ].\label{eq:iterations}
\end{multline}

In Table \ref{alg3} 
 we report the function Fast Weighted Split Bregman (FWSB)  for the solution of problem \eqref{B}.
The  output variable  $U^{(j)}$ is the computed solution and $\bar m$ is the number of total  iterations.
\begin{table}[!htbp]
\begin{algo}
\begin{center}
\begin{tabular}{c}\hline
\textbf{[$U^{(j)}$, $\bar m$]=FWSB($\lambda, \theta, \beta,{v}^{(n)},w^x,w^y$)}\\
 \hline 
\begin{minipage}[c]{11cm}
\begin{tabular}{ll}
$\Lambda=\frac{\lambda}{\theta}$ , $\bar m=0$, \\
 $U^{(0)}={v}^{(n)}$, $e_x^{(0)}=e_y^{(0)}=0$;\\
\quad  $j=1$\\
\quad \quad {\it repeat } &\\
\quad \quad \quad $U_x={\nabla_x^w} U^{(j-1)}$; $U_y={\nabla_y^w} U^{(j-1)}$; & \\
\quad \quad \quad $X_x^{(0)}=U_x$; $X_y^{(0)}=U_y$ & \\
 \quad  \quad \quad$z_x= U_x+e_x^{(j-1)}$; $z_y= U_y+e_y^{(j-1)}$; &\\
 \quad \quad \quad $e_x^{(j)}=Cut_{\Lambda}(z_x)$; $e_y^{(j)}=Cut_{\Lambda}(z_y)$; &\\
 \quad\quad \quad $m=0$\\
\quad \quad \quad {\it (Solution of problem \eqref{Slin1}) }\\
\quad \quad \quad {\it repeat } &\\
 \quad \quad \quad \quad $X^{(m+1)}={v}^{(n)}- \beta \theta\left({\nabla_x^w}^T(X_x+2 \cdot e_x^{(j)}-z_x)+
{\nabla_y^w}^T (X_y +2 \cdot e_y^{(j)}-z_y)\right)$&\\
 \quad   \quad \quad \quad $X_x={\nabla_x^w} X^{(m+1)}$; $X_y={\nabla_y^w} X^{(m+1)}$&\\
 \quad  \quad \quad \quad  $m=m+1$&\\
 \quad  \quad \quad {\it until {\tt  stopping condition \eqref{eq:stop_sb}}}&\\
 \quad \quad \quad$U^{(j+1)}=X^{(m)}$;&\\
 \quad \quad \quad $\bar m=\bar m+m;j=j+1$&\\
\quad \quad {\it until {\tt stopping condition \eqref{eq:stop_sb}}}&\\
\end{tabular}
\end{minipage} \\ \hline
\end{tabular}
\end{center}
 \label{alg3}
\end{algo} 
\caption{ FWSB Algorithm for the solution of problem \eqref{B}}
\end{table}

The  stopping condition of both the loops (with indices $j$ and $m$) is defined on the basis of the relative tolerance parameter $\tau$ as follows:
\begin{equation}
\|w^{(k+1)}-w^{(k)}\| \leq \tau \| w^{(k))} \|
 \label{eq:stop_sb}
\end{equation}
where $w^{(k)} \equiv U^{(j)}$ in the outer loop ($k\equiv j$)  and $w^{(k)} \equiv X^{(m)}$ in the inner loop ($k\equiv m$).

\section{Numerical Experiments \label{NumRes}}
In this section we analyze the results obtained by applying Algorithm \ref{alg2} to two representative test problems related to different image processing applications. The minimization problem \eqref{eq:convex} is solved applying the accelerated Forward Backward method together with the weighted split Bregman method. Our aim is to compare the proposed FWSB method with  Gauss Seidel (WSB\_GS) applied  to the linear system \eqref{eq:matr}.
The experiments are performed on a PC intel i7 with 32 Gbyte Ram, by using Matlab R2018a.

In test problem T1, we consider an image deblurring problem where the matrix $\Phi$ in \eqref{eq:convex} is a Gaussian blur operator obtained by the Matlab function {\tt fspecial}, with standard deviation $\sigma=1.5$ and size $9$. The algorithms are tested  both on noiseless and noisy data. 
The results reported here are relative to the case of Gaussian white noise with variance $\delta=0.5 \cdot 10^{-2}$.

Test problem T2 is a compressed sensing application, where $z$ represents  subsampled Magnetic Resonance data in the so called  Kspace and the matrix $\Phi$ in \eqref{eq:model} is the undersampled Fourier matrix,  obtained by the Hadamard  product between  the full resolution Fourier matrix $\mathbf{F}$ 
and the mask $\mathcal{M}$, i.e.
\begin{equation}\Phi = \mathcal{M} \circ \mathbf{F}.
\label{eq:mask}
\end{equation}
In the tests reported in the present work  we consider $\mathcal{M}$ as a radial mask with sampling percentages $S_p=3.98 \%$ and $S_p=4.3 \%$ relative to $8$ and $10$ radial lines respectively.

The quality of the reconstructed  image  is evaluated by means of the  Peak Signal to Noise Ratio (PSNR)
$$ PSNR=20 \log_{10}   \frac{\max(x)}{rmse}, \ \hbox{where} \ \ rmse=\sqrt{\frac{\sum_i \sum_j (u_{i,j}-x_{i,j})^2}{N^2}}. $$
where $x$ is the reference true image and $u$ is the reconstructed image.
Since our purpose is to evaluate the best possible solution obtained by each method, in all tests  the regularization parameter $\lambda$ is heuristically set to the best possible value with respect to PSNR.

\begin{table}[h!]
\centering
\begin{tabular}{c ccccc }
\hline
\multirow{2}{*}{Test} & \multirow{2}{*}{{\tt Par}}& \multicolumn{2}{c}{FWSB} & \multicolumn{2}{c}{WSB\_GS}\\
\cline{3-6}
                      &  & PSNR & time & PSNR & time   \\
\hline
\multirow{2}{*}{T1} & $\delta = 0 $  & $25.5$ & $7.20s$ & $25.27$ & $50.64s$\\
                    & $\delta = 0.5 \cdot 10^{-2}$ & $24.38$ & $5.74s$ & $24.29$ & $15.17s$\\
\hline
\multirow{2}{*}{T2} & $S_p=4.30$ & $36.22$  & $3.36s$ & $33.32$ & $12.64s$ \\
                    & $S_p=3.98$  & $28.91$ & $5.25s$ & $27.86$ & $17.94s$ \\
\hline
\end{tabular}
\caption{Values of PSNR and times for the different tests. The column {\tt Par} reports the parameters used in the tests:the variance noise ($\delta$) for T1 and the  sampling percentage ($S_p$) for T2. }
\label{tab:T1}
\end{table}
In table \ref{tab:T1} we report the PSNR and computation times obtained by the two test problems. Column {\tt Par} shows the parameters of each experiment: noise variance in case of deblur (T1), Sampling percentage in case of MRI (T2). 
We observe that FWSB is always the most efficient obtaining smaller computational times. Regarding the accuracy,
we can see that  FWSB always reaches the greatest values of PSNR. 
In figures \ref{fig:PSNR_time} and \ref{fig:PSNR_time_1} we can appreciate the evolution of PSNR and computation times after each FB iteration
for both FWSB and WSB\_GS, we remark the better performance FWSB in terms of accuracy and computation times.

\begin{figure}
  \begin{center}
\begin{tabular}{cc}    
      \includegraphics[width=2.2in]{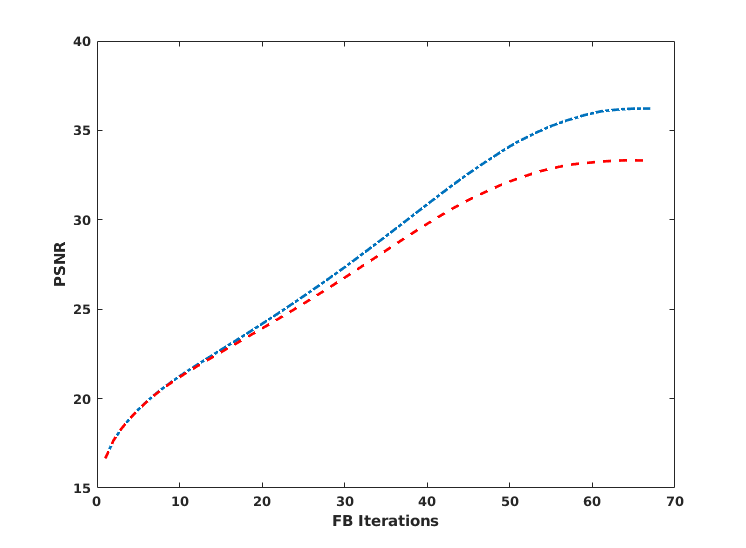} & 
			\includegraphics[width=2.2in]{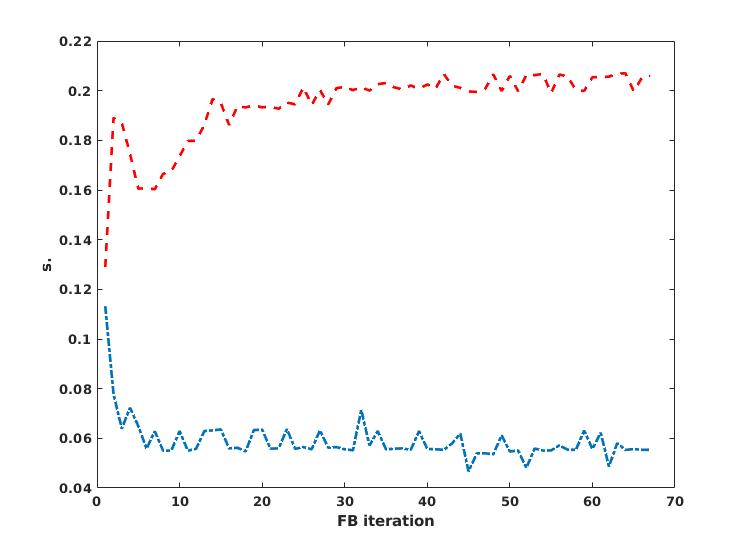} \\
			(a)  & (b)  \\
\end{tabular}
\caption{T2 test $S_p=4.30$, FWSB, blue dash-dot line; WSB\_GS, red dashed line. (a) PSNR vs. FB iterations. (b) Time in seconds(s.) vs FB iterations.}
  \label{fig:PSNR_time}
  \end{center}
\end{figure}

\begin{figure}
  \begin{center}
\begin{tabular}{cc}    
      \includegraphics[width=2.2in]{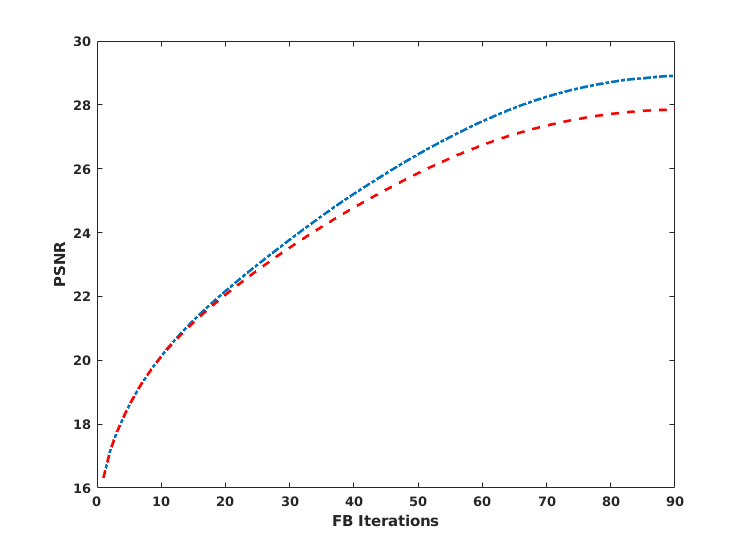} & 
			\includegraphics[width=2.2in]{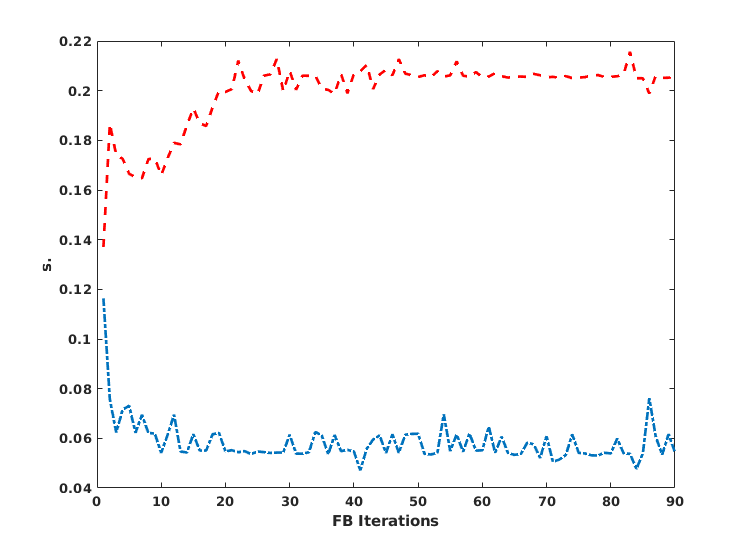} \\
			(a)  & (b)  \\
\end{tabular}
\caption{T2 test $SP=3.98$:  FWSB, blue dash-dot line; WSB\_GS, red dashed line. (a) PSNR vs. FB iterations. (b) Time in seconds(s.) vs FB iterations.}
  \label{fig:PSNR_time_1}
  \end{center}
\end{figure}

\section{Conclusions \label{Con}}

In this work we proposed a fast splitting Method for the solution of the inner step of the Weighted Split Bregmann method.
We proved its convergence and compared it to the most commonly used iterative methods.
After running a large set of experiments for different problems and datasets, we reported the most representative results obtained in the case of image deblurring and sparse MRI. From the results we can state that WFSB is the most efficient and accurate  method to be used
in the solution of Weighted Split Bregman Method.

\section*{References}

\bibliography{biblio}

\end{document}